\documentclass[14pt]{extarticle}

\usepackage[utf8]{inputenc}
\usepackage[english]{babel}

\usepackage{amsthm,amssymb,amsopn,amsmath}
\usepackage{indentfirst}
\usepackage{graphicx}

\usepackage{textcomp}
\newcommand{\ftextnumero}{{\fontfamily{txr}\selectfont \textnumero}}

\usepackage{mathpazo,eulervm}  

\usepackage[top=1.5cm,bottom=1.5cm,left=1.5cm,right=1.5cm]{geometry}

\renewcommand{\le}{\leqslant}
\renewcommand{\ge}{\geqslant}

\newcommand{\R}{\mathbb{R}}

\newcommand{\eqdef}{\stackrel{\mathrm{def}}{=}}

\DeclareMathOperator{\dom}{\mathrm{dom}}
\DeclareMathOperator{\epi}{\mathrm{epi}}
\DeclareMathOperator{\conv}{\mathrm{conv}}
\DeclareMathOperator{\Int}{\mathrm{int}}
\DeclareMathOperator{\dist}{\mathrm{dist}}
\DeclareMathOperator{\toe}{\stackrel{e}{\to}}

\DeclareMathOperator*{\Li}{\mathrm{Li}}
\DeclareMathOperator*{\Ls}{\mathrm{Ls}}

\newtheorem{thm}{Theorem}
\newtheorem{lemma}{Lemma}

\theoremstyle{remark}
\newtheorem{remark}{Remark}

\theoremstyle{definition}
\newtheorem{defn}{Definition}

\title{Hessian Measures in the Aerodynamic\\Newton Problem}
\author{L.~V.~Lokutsievskiy and M.~I.~Zelikin\footnote{This work is supported by the Program of the Presidium of the Russian Academy of Sciences \ftextnumero 01 'Fundamental Mathematics and its Applications' under grant PRAS-18-01 and by the Russian Foundation for Basic Research under grants 17-01-00805 and 17-01-00809.}}
\date{}

\begin{document}

\maketitle

\begin{abstract}
    Simple natural proofs of all known results regarding the aerodynamic Newton problem are obtained. Additional new theorems and new promising formulas in terms of Hessian measures are found.
\end{abstract}


\section{Introduction}

In 1685 Sir Isaac Newton posed and solved a problem on the profile of a body that gives minimal resistance to motion in a rare medium. Newton reasoned that the profile should be a surface of revolution of a curve $z=f(x)$ about the vertical $Oz$-axis. Suppose that the body moves vertically downward. In this case, the resistance is defined by the integral       

\begin{equation}
	\label{eq:fist_newton_problem}
	\int_0^{x_0}\frac {xdx}{1 + u'(x)^2}
\end{equation}

\noindent with conditions $u(0)=0$ and $u(x_0)=M$.

It is supposed that any particle meets the body just ones, and thus the solution should be found in the class of convex bodies. It is natural to take $v=u'(x)$ as a control with constraint $v \geq 0$. The constraint is suited to the requirement of convexity of the surface. In the absence of this condition, the lower bound of the functional \eqref{eq:fist_newton_problem} equals zero. Indeed, it may take arbitrarily small positive values if we consider sharp infinitely big oscillations of the control $v$ (slopes of the trajectory $u(x)$). Oscillations of this type are incompatible with the physical conditions under which the functional was derived (the condition of the single collision of a particle with a body).

Nevertheless, some strange phenomena that appear in dropping the convexity condition have been investigated. In particular, in 2009, Aleksenko and Plakhov found a body (biplane of Busemann) that is not convex and does not change incident rays of light due to multiple reflection. Hence this body has zero resistance and appears invisible (see \cite{Plakhov}). 

The Newton problem is considered as the first solved problem of optimal control because the convexity condition is an inequality-type condition. The Newtonian solution is refined. It turns out that the optimal control on the front part of the surface is $v \equiv 0$, i.e. the nose of the optimal surface must be truncated by a flat front part and cannot be sharp. The equation for the remaining side part admits the following explicit expression in parametric form ($p_0<0$):

\[
	\begin{cases}
		u=-\frac{p_0}{2}\left( \ln \frac{1}{v}+v^2+\frac34 v^4-\frac74\right); \\
		x=-\frac{p_0}{2}\left(\frac{1}{v}+2v+v^3\right).\\
	\end{cases}
\]

\noindent It was proved a long time ago that this solution is optimal in the class of convex bodies of revolution. The proof can be found in textbooks on optimal control theory and calculus of variations.

It was realized only at the very end of 20-th century that the statement of the problem 
\eqref{eq:fist_newton_problem} is not fully adequate. The point is that the resistance is defined by the multiple integral

\begin{equation}
	\label{eq:2dim_functional_intro}
	J = \int_{\Omega}\frac {dx_1 \wedge dx_2}{1+|u'|^2}.
\end{equation}

\noindent In 1995, Guasoni and Buttazzo constructed a counterexample in the case where $\Omega$ is a circle, which shows that the double integral \eqref{eq:2dim_functional_intro} takes smaller values on convex surfaces that are not surfaces of revolution (see \cite{Guasoni}). Marcellini \cite{Marcellini1990} proved the existence theorem for the functional \eqref{eq:2dim_functional_intro}. Nevertheless, no solution of this classical problem has been known until now. The solution is not unique due to the abandonment of the axis-symmetry condition.

The functional \eqref{eq:2dim_functional_intro} and the corresponding Euler equation

\begin{equation}
	\label{eq:euler_lagrange_equation}
	\mathrm{div}\,\frac{u'}{(1+|u'|^2)^2} = 0
\end{equation}

\noindent are similar to those for minimal surfaces:

\[
	\int_\Omega \sqrt{1+|u'|^2}dx_1\wedge dx_2\to\min
	\quad \Longrightarrow \quad
	\mathrm{div}\,\frac{u'}{\sqrt{1+|u'|^2}} = 0.
\]

There exists a deeper similarity. The mean curvature of minimal surfaces equals zero and it was proved in \cite{Buttazzo1995} that the extremals of the functional \eqref{eq:2dim_functional_intro} have zero Gaussian curvature and are developable on their smooth parts. A striking fact is that tens of minimal surfaces are known but we still do not know examples minimizing the functional \eqref{eq:2dim_functional_intro}.

\section{Statement of the problem}

Let us give an exact statement of the aerodynamic Newton problem in terms of convex analysis. Let
$\R^n$ be the Euclidean space, let $\Omega\subset\R^n$ be a convex compact set with nonempty interior, and let
$M\ge 0$. Denote by $C_M$ the set of closed convex functions $u:\R^n\to \R$ such that $\dom u = \Omega$, $\inf u=0$, and $u|_{\Omega}\le M$. The Newton problem is to find a minimum on $C_M$ of the functional

\begin{equation}
\label{problem:start}
	J(u) = \int_{\Omega}\frac{1}{1+|u'(x)|^2}\,dx \to \min,\quad u\in C_M.
\end{equation}

\noindent (note that the derivative $u'(x)$ of a convex function exists almost everywhere).

The functional $J$ defines the resistance of a convex $(n+1)$-dimensional rigid body with the form $\epi u$ (or $\epi u\cap\{z\le M\}$) in a constant vertical rarefied flow of particles (moving upward). The first interesting case of a three-dimensional body arises when $n=2$.

In our opinion, at present, there are four most significant results on the structure of optimal solutions in the class $C_M$:

\begin{enumerate}
	\item \textit{The existence of a solution.} There exists an optimal solution of problem \eqref{problem:start}.
	\item \textit{Modulus of the gradient.} If (i) $u\in C_M$ is an optimal solution, (ii) there exists $u'(x)$ at a point $x\in\Int\dom u$, and (iii) $0<u(x)<M$, then $|u'(x)|\ge 1$ (see~\cite{Buttazzo1995}).
	\item \textit{Lack of strict convexity.} If $\omega\subset\Omega$ is an open set and $u|_\omega\in C^2$, then $\det u''(x)=0$ for any $x\in\omega$ (see \cite{LachandCurvature}).
	\item \textit{The front part of the body.} Let $n=2$, and let the set $\Omega$ be the unit circle. Consider the minimization problem on a smaller class\footnote{The index $h$ in $C_M^h$ denotes the word ``heel''.} $C_M^h\subset C^M$  of functions with a smooth side boundary. Namely, $C_M^h=\{u\in C_M:u|_{\partial \Omega}=M, u\in C^2\big(\Omega\setminus u^{-1}(0)\big)\}$. Then the front side $u^{-1}(0)$ of any optimal solution in $C_M^h$ is a convex regular polygon (see~\cite{LachandPolygon}).
\end{enumerate}

\begin{remark}
\label{rm:non_empty_heel}
According to result 2, the front part $u^{-1}(0)$ of the body must contain all points $x\in\Omega$ such that $\dist(x,\partial\Omega)\ge M$.
Therefore, if

	\[
		M<\max_{x\in\omega} \dist\,(x,\partial\Omega),
	\]

	\noindent then the front part $u^{-1}(0)$ of the body has a nonempty interior.	
\end{remark}

The most astonishing result 3 means that any optimal body does not have strictly convex smooth parts of the boundary (it may explain why the classical Newtonian solution is non-optimal in the class 
$C_M$). So any small variation of an optimal function $u$ by a $C^\infty$-function may take $u$ out of the class $C_M$. Therefore, despite the fact that the problem looks like a variational one, it is not.
For instance, the optimal solution in the class $C_M$ does not need to meet the Euler equation~\eqref{eq:euler_lagrange_equation}.

In this work, we propose a unified approach to the study of the aerodynamic Newton problem using the Legendre transformation and Hessian measures (see \cite{ColesantiFirst}). Using this approach, we obtain all the four above-mentioned results and also

\begin{itemize}
	\item prove a theorem on the passage to the limit in integral \eqref{problem:start} requiring only the pointwise convergence $u_k\to u$ at the interior part of $\Omega$ (see \S \ref{sec:pointwise_limit});	
	\item investigate the properties of optimal solutions in a non-explored class $C_M^m\subset C_M$ of convex functions with a developable side boundary having a unique Maxwell stratum (see. 
	\S \ref{sec:maxwell}).
\end{itemize}

\section{Legendre's transformation}
\label{sec:Legandre_transform}

The main idea of our approach is to change the coordinates $x\mapsto p(x)$ in the integral \eqref{problem:start}, where the change $x\mapsto p(x)$ is obtained from the Legendre transformation of a convex function $u$. Generally speaking, the classical Legendre transformation $u^*(p)=\sup_{x}\big(\langle p,x\rangle - u(x)\big)$ defines the mapping $x\mapsto p=u'(x)$ only if $u\in C^1$. In the general case $u\not\in C^1$, the mapping $x\mapsto p$ is multivalued and the direct change is impossible. Nevertheless, the described change of coordinates in the integral \eqref{problem:start} is possible due to Colesanti and Hug \cite{ColesantiFirst,ColesantiHug}, since the Lebesgue measure $L^n$ on $\R^n$ turns into the Hessian measure $F_0$ on $\R^{n*}$ defined by the conjugate function~$u^*$.

Let us give a short clarification to the Hessian measures. In the work \cite{ColesantiFirst} it was proved that a convex function (in our case $u^*$) defines a collection of measures $F_j$, $j=0,\ldots,n$, on $n$-dimensional Euclidean space as follows. Let $\eta\subset \R^{n*}$ be a Borel set.
Let us define\footnote{Here we use the canonical isomorphism $\R^n\simeq\R^{n*}$, since $\R^n$ is Euclidean in the Newton problem.}

\[
	\eta^\varepsilon = \bigcup_{p\in\eta} \big(p+\varepsilon\partial u^*(p)\big).
\]

\noindent Then the volume of $\eta^\varepsilon$ is a polynomial in $\varepsilon$, i.e. the following analogue of the Steiner formula is fulfilled \cite{Schneider}:

\[
	L^n(\eta^\varepsilon) = \sum_{j=0}^n \binom{n}{j} F_{n-j}(\eta|u^*) \varepsilon^{j},
\]

\noindent where $L^n$  is Lebesgue measure.

\begin{defn}
	The measures $F_j(\,\cdot\,|u^*)$ on $\R^{n*}$ being the coefficients of the polynomial $L^n(\eta^\varepsilon)$ defined by a convex function $u^*$, are called \textit{Hessian measures} of the function $u^*$.
\end{defn}

A general construction of the Hessian measures can be found in \cite{ColesantiHug}. Let us remark that if
$u^*\in C^2(\omega^*)$ on a domain $\omega^*\subset \R^{n*}$, then, for any $\eta\subset\omega^*$, we have

\begin{equation}
\label{eq:F_j_S_j}
	\binom{n}{j} F_{j}(\eta|u^*) = \int_\eta S_{n-j}(p) dp,
\end{equation}

\noindent where $S_j(p)$ denotes the elementary symmetric polynomial of degree $j$

\[
	S_j(p) = \sum_{1\le k_1 < \ldots<k_j\le n} \lambda_{k_1}(p)\ldots\lambda_{k_j}(p)
\]

\noindent of the eigenvalues $\lambda_1(p),\ldots,\lambda_n(p)$ of the Hessian form $(u^*)''(p)$.

It is easy to see that $F_n\equiv L^n$. Moreover, it was proved in \cite{ColesantiFirst} that
 
\begin{equation}
\label{eq:F_0_and_L_n}
	F_0(\eta|u^*) = L^n\Big( \bigcup_{p\in\eta}\partial u^*(p) \Big) = 
	L^n\Big\{ x: \partial u(x)\cap\eta\ne\emptyset \Big\}.
\end{equation}

\noindent (More general relationships between Hessian measures of a function and its conjugate can be found in  \cite[Theorem 5.8]{ColesantiHug})

We have

\[
	J(u) = \int_{\R^{n*}} \frac{1}{1+|p|^2} F_0(dp|u^*) \eqdef J^*(u^*).
\]

Let us determine the restrictions to $u^*$. The condition $0=\inf_{x\in\Omega}u(x) = \inf_{x\in\R^n}u(x)$ is equivalent to $u^*(0)=0$.

Let us denote by $s_\Omega$ the supporting function of the set $\Omega$ and calculate

\[
	\sup_{x\in\Omega}u(x) = \sup_{x\in\Omega,p\in\R^{n*}}\big(\langle p,x\rangle - u^*(p) \big) = 
	\sup_{p\in\R^{n*}}\big( s_\Omega(p) - u^*(p)\big).
\]

\noindent So the condition $u|_\Omega\le M$ is equivalent to $u^*(p)\ge s_\Omega(p)-M$ for all $p\in\R^{n*}$.

It remains to reformulate the condition $\dom u = \Omega$ in terms of $u^*$. The condition $\dom u\supset\Omega$ follows from $u|_\Omega\le M$, i.e. from $u^*(p)\ge s_\Omega(p)-M$, and the condition
$\dom u\subset\Omega$ is equivalent to
$\partial u^*(p)\subset \Omega$ for all $p\in\R^{n*}$.

\begin{defn}
Let us denote by $C_M^*$ the set of all closed convex functions $u^*$ on $\R^{n*}$ that satisfy the three following conditions:

	\[
		(i)\,u^*(0) = 0;
		\quad
		(ii)\,u^*(p)\ge s_\Omega(p)-M;
		\quad
		(iii)\,\partial u^*(p)\subset \Omega
		\qquad \forall p\in\R^{n*}.
	\]
\end{defn}

Thus, we have proved the following statement

\begin{thm}
\label{thm:Legandre_transform}
	\[
		u\in C_M \ \Longleftrightarrow\  u^*\in C_M^*
		\qquad\mbox{and}\qquad
		J(u)=J^*(u^*).
	\]

\end{thm}

\begin{remark}
	It is easy to see that if $u^*\in C_M^*$, then $u^*\le s_\Omega$. Indeed,
	
	\[
		u^*(p) = \sup_{x\in\Omega}\big(\langle p,x\rangle - u(x)\big) \le
		\sup_{x\in\Omega}\langle p,x\rangle - \inf_{x\in\Omega}u(x) = s_\Omega(p).
	\]
	
	\noindent In particular, since $\Omega=\dom u$ is a compact set, we have $\dom u^*=\R^{n*}$.
	
\end{remark}

The aerodynamic Newton problem in the dual space $\R^{n*}$ is formulated as follows:

\[
	J^*(u^*)\to\min,\qquad u^*\in C_M^*.
\]

\section{Passage to the limit}
\label{sec:pointwise_limit}

We show that the continuity of the functional $J$ relative to the pointwise convergence functions in $C_M$ follows from Theorem \ref{thm:Legandre_transform}. This result may also serve as a confirmation of numerical methods, because it proves that any function near an optimal solution gives the value of the functional $J$ close to the optimal.

\begin{thm}
\label{thm:pointwise_limit}
	Suppose that a sequence $u_k\in C_M$ converges pointwise on the interior $\Int\Omega$ to a function $u\in C_M$. Let $f:\R^{n*}\to\R$ be a bounded continuous function. Then

	\[
		\lim_{k\to\infty}\int_{\Omega}f(u_k'(x))\,dx =
		\int_{\Omega}f(u'(x))\,dx,
	\]

	\noindent and all these integrals are defined and finite.
	
\end{thm}

The proof of the theorem is based on the fact that the Legendre transformation is bi-continuous relative to the convergence of epigraphs of convex functions in the sense Kuratowski. Let us give the necessary definitions.

\begin{defn}
	Let $A_k\subset \R^n$  be a sequence of sets. The sets

	\[
		\Li_{k\to\infty} A_k = \{x\in\R^n:\limsup_{k\to\infty}\dist(x,A_k)=0\}
	\]

	\noindent and
	
	\[
		\Ls_{k\to\infty} A_k = \{x\in\R^n:\liminf_{k\to\infty}\dist(x,A_k)=0\}.
	\]
	
	\noindent are called the lower and the upper Kuratowski limits of the sequence $A_k$, respectively.
\end{defn}

It is easy to see from the definition that the sets $\Li_{k\to\infty} A_k$ and $\Ls_{k\to\infty} A_k$ are closed, $\Li_{k\to\infty} A_k\subset \Ls_{k\to\infty} A_k$, and if all sets $A_k$ are convex, then the set $\Li_{k\to\infty} A_k$ is convex too.

\begin{defn}
	It is said that a sequence of closed convex functions $u_k$ on $\R^n$ converges in the sense of epigraphs to a closed convex function $u$ on $\R^n$ if

	\[
		\Li_{k\to\infty} \epi u_k = \Ls_{k\to\infty} \epi u_k = \epi u.
	\]

	\noindent In this case, one writes  $u_k\toe u$.
\end{defn}

To prove Theorem \ref{thm:pointwise_limit}, we shall use the be-continuity of the Legendre transformation relative to the convergence of epigraphs. Namely, if $u_k$ and $u$ are closed convex proper functions on $\R^n$, then $u_k\toe u$ iff $u^*_k\toe u^*$ \cite{Mosco,Joly}. We need the following

\begin{lemma}
\label{lm:pointwise_toe}
	Let $u_k$ and $u$ be closed convex functions on $\R^n$ with common effective domain $\dom u_k=
\dom u=\Omega$, and let $\Int\Omega\ne\emptyset$. Suppose that $|u_k|\le M$ and $|u|\le M$ on $\Omega$. Then if  $u_k(x)\to u(x)$ for any $x\in\Int\Omega$, then $u_k\toe u$.
\end{lemma}
	
\begin{proof}
	To prove the lemma, we need to verify that
		
	\begin{equation}
	\label{eq:lm_3_subsets}
		\epi u \subset \Li_{k\to\infty}\epi u_k\subset \Ls_{k\to\infty}\epi u_k\subset \epi u.
	\end{equation}

	\noindent First, we remind that the intermediate inclusion follows from the definition of the Kuratowski limit. 
	
	Second, let us prove the first inclusion in \eqref{eq:lm_3_subsets}. We begin with the interior of $\epi u$:
	
	\[
		\Int\epi u = \{(x,a):x\in\Int\Omega\mbox{ and }u(x)<a\}.
	\]

	\noindent Let $(x,a)\in\Int\epi u$. Then $u_k(x)\to u(x)$ by the condition of the lemma.  Hence  $u_k(x)<a$ for the big enough $k$ and, consequently,  $(x,a)\in\Li_{k\to\infty}\epi u_k$. It remains to note that $\epi u$   is a convex set with a nonempty interior, and so
	
	\[
		\epi u \subset \mathrm{cl}\,\Int\epi u \subset 
		\mathrm{cl}\,\Li_{k\to\infty}\epi u_k = \Li_{k\to\infty}\epi u_k,
	\]

	\noindent since the set $\Li_{k\to\infty}\epi u_k$ is closed.
	
	It remains to prove the last inclusion in \eqref{eq:lm_3_subsets}. Consider a point $(x,a)\in\Ls_{k\to\infty}\epi u_k$. There are two possible cases.

	The first case is $x\in\Int\Omega$. Consider a neighborhood $\omega$ of the point $x$ compactly embedded in $\Omega$, $\omega\Subset\Omega$. In this case, all the functions $u_k$ and $u$ are Lipschitz on $\omega$ with the same Lipschitz constant $2M/\dist(\omega,\partial\Omega)$. Therefore, the sequence $u_k$ converges to $u$ uniformly on $\omega$. Since $(x,a)\in\Ls_{k\to\infty}\epi u_k$, there exists a sequence\footnote{The sequence of indices $k_j$ tends to infinity monotonically.}  $k_j$ and points $x_{k_j}\in\omega$ such that $x_{k_j}\to x$, and there exists $\lim_{j\to\infty} u_{k_j}(x_{k_j})\le a$. Therefore, the uniform convergence $u_{k_j}\rightrightarrows u$ on $\omega$ and the continuity of $u$ in $x$ give
	
	\[
		|u(x)-u_{k_j}(x_{k_j})| \le |u(x)-u(x_{k_j})| + \|u-u_{k_j}\|_{C(\omega)} 
		\stackrel{j\to\infty}{\longrightarrow} 0.
	\]

	\noindent Hence $u(x)\le a$ and $(x,a)\in\epi u$.
	
The second case is $x\in\partial\Omega$. Let us first prove that the convex hull of the point $(x,a)$ and the interior of $\epi u$ belongs to the upper limit set:  
   
	\begin{equation}
	\label{eq:lm_conv_in_Ls}
		\conv \big((x,a),\Int\epi u\big) \subset \Ls_{k\to\infty}\epi u_k.
	\end{equation}

	\noindent Take a sequence of indices $k_j$ and points $x_{k_j}\in\Omega$ such that $x_{k_j}\to x$ and there exists $\lim_{j\to\infty}u_{k_j}(x_{k_j})\le a$. Let $(y,b)\in\Int\epi u$. Then $(y,b)\in\epi u_{k}$ for a enough large $k$, since $u_k(y)\to u(y)<b$. Therefore, for any number $\lambda\in[0;1]$ and for a large enough $k_j$, one has 

	\[
		(\lambda x_{k_j}+(1-\lambda)y,\lambda u_{k_j}(x_{k_j}) + (1-\lambda)b)\in\epi u_{k_j}
	\]

	\noindent and, consequently, $(\lambda x+(1-\lambda)y,\lambda a+(1-\lambda)b)\in\Ls_{k\to\infty}\epi u_k$, which proves \eqref{eq:lm_conv_in_Ls}.

	Let us now prove that the point $(x,a)$ belongs to the closure of $\epi u$ and then use the closeness of $\epi u$ (required in the statement of the lemma). So let us find a point from $\epi u$ in any neighborhood $N$ of the point $(x,a)$. Because the convex hull of \eqref{eq:lm_conv_in_Ls} has a nonempty interior, it intersects with $N$. Take any point $(x_1,a_1)$ in this intersection. It is evident that $(x_1,a_1)\in\Ls_{k\to\infty} \epi u_k$ due to \eqref{eq:lm_conv_in_Ls}. Since $x_1\in\Int\Omega$, it follows that (according to the first case) $(x_1,a_1)\in\epi u$, as required.
\end{proof}

\begin{proof}[Proof of theorem \ref{thm:pointwise_limit}.]

Let us use the following fact. For any convex function $v\in C_M$, we have

	\[
		\int_\Omega f(v'(x))\,dx = \int_{\R^{n*}} f(p)F_0(dp,v).
	\]

	\noindent Note that the integral on the right hand side is well defined for any bounded continuous function $f$, since $F_0(\R^{n*},v)=L^n(\Omega)<\infty$. 

	First, we prove the theorem in the case where the function $f$ has a compact support. Let the sequence $u_k\in C_M$ on $\Int\Omega$ converges pointwise to a function $u\in C_M$. Then, by Lemma \ref{lm:pointwise_toe}, the sequence $u_k$ converge to $u$ in the sense of Kuratowski convergence of epigraphs. The Legendre transformation is bi-continuous relative to this convergence \cite[Theorem 1]{Mosco}; therefore, $u_k^*\stackrel{e}{\to}u^*$. It means that the sequence $u_k^*$ converges pointwise to $u^*$ on $\Int\dom u^*=\R^{n*}$ \cite[Corollary~3C]{Wets}. It remains to see that the pointwise convergence gives us the convergence $F_j(\,\cdot\,|u_k^*)\to F_j(\,\cdot\,|u^*)$ in the topology of the space $C_c^*$, which is dual to the space $C_c$ of continuous functions with compact support \cite[Theorem~1.1]{ColesantiHug}.

	Now let $|f|\le A$ be an arbitrary bounded continuous function. Let us approximate $f$ by functions with compact support. Let $\psi_j:\R^{n*}\to \R$ be a continuous function with $\psi_j(p)=1$ for $|p|\le j$, $\psi_j(p)=0$ for $|p|\ge j+1$ and $0\le\psi_j(p)\le 1$ for all $p$. Then
		
	\begin{multline}
	\label{eq:thm_poinwise_limit_leq_sum}
		\left|\int_{\R^{n*}}f\,F_0(dp,u) - \int_{\R^{n*}}f\,F_0(dp,u_k)\right| \le
		\left|\int_{\R^{n*}}\psi_j f\big[F_0(dp,u) - F_0(dp,u_k)\big]\right| + \\
		+\left|\int_{\R^{n*}}(1-\psi_j)f\,F_0(dp,u)\right| +
		\left|\int_{\R^{n*}}(1-\psi_j) f\,F_0(dp,u_k)\right|.
	\end{multline}

	Let us estimate each of the three summands on the right-hand side of inequality \eqref{eq:thm_poinwise_limit_leq_sum}. The second and third summands are estimated identically: let $v\in C_M$; then, using \eqref{eq:F_0_and_L_n}, we obtain

	\[
		\left|\int_{\R^{n*}}(1-\psi_j)f\,F_0(dp,v)\right| \le
		A\int_{|p|\ge j}\,F_0(dp,v) =
		A L^n\big\{x:\exists p\in\partial v(x),\ |p|\ge j\big\}.
	\]

	We claim that the measure of the set in right-hand side tends to zero as $j\to\infty$. Indeed, if a convex function $v\in C_M$ has a subgradient at a point $x$ which is greater than $j$ in absolute value, then $x$ belongs to a neighborhood of the boundary $\partial \Omega$ of radius $M/j$. The volume of the neighborhood tends to zero as $j\to\infty$, since the compact set $\Omega$ is convex. Therefore, for any $\varepsilon>0$, there exists a number $j_0$ such that, for all $v\in C_M$, the following inequality holds:
	
	\[
		\left|\int_{\R^{n*}}(1-\psi_{j_0})f\,F_0(dp,v)\right| <\varepsilon.
	\]

	Let us fix $j=j_0$ and estimate the first summand in \eqref{eq:thm_poinwise_limit_leq_sum}. It tends to zero as $k\to\infty$, since the function  $\psi_{j_0} f$ has a compact support. Therefore, one can find a number $k_0$ such that the following inequality holds for all $k>k_0$:
 
	\[
		\left|\int_{\R^{n*}}\psi_{j_0} f\,F_0(dp,u) - \int_{\R^{n*}}\psi_{j_0} f\,F_0(dp,u_k)\right| < \varepsilon.
	\]

	\noindent We obtain the desired result by substituting the previous inequality in \eqref{eq:thm_poinwise_limit_leq_sum}.
\end{proof}

Let us note that the known result \cite{Marcellini1990,Buttazzo1995} on the existence of an optimal solution in problem \eqref{problem:start} is based on the lower semi-continuity of the functional $J$ with respect to the strong topology in $W_{\mathrm{loc}}^{1,p}(\Omega)$. Theorem \ref{thm:pointwise_limit} a gives better result under a less stringent assumption for the case $f(x,u,u')=f(u')$. Moreover, we can easily prove the existence theorem using Theorem \ref{thm:pointwise_limit}. Indeed, the functional $J$ is continuous relative to pointwise convergence on $\Int\Omega$; therefore, it is sufficient to prove that there exists a subsequence of a minimizing sequence that converges pointwise on $\Int\Omega$.

\begin{thm}
\label{thm:C_M_compact}
	Let $u_k\in C_M$ be a sequence. There exists a subsequence $u_{k_j}$ that converges pointwise on $\Int\Omega$.
\end{thm}

\begin{proof}
Let $\omega_1\subset\omega_2\subset\ldots\Subset\Omega$ be a compact exhaustion of $\Omega$. Then the functions $u_k$ are Lipschitz on any set $\omega_j$ with the same Lipschitz constant $M/\dist(\omega_j,\partial\Omega)$. Let us use the Arzel\`a--Ascoli lemma and select a subsequence that converges uniformly in~$\omega_1$; then select a subsequence that converges uniformly in $\omega_2$, and so on. So we apply the standard argument of choosing a diagonal subsequence and obtain a subsequence $u_{k_j}$ which converges uniformly on any set compactly embedded in $\Omega$. Therefore, the subsequence $u_{k_j}$ converges pointwise on $\Int\Omega$, as is required to be proved.

\end{proof}

\section{The modulus of the gradient.}

Here we propose a new proof of result 2 on the modulus of the gradient that is based on the use of Hessian measures. It will be shown that

\begin{equation}
\label{eq:grad_u_not_in_01}
	\mbox{a function }u\in C_M\mbox{ is optimal}
	\quad\Longrightarrow\quad
	|u'(x)|\not\in]0;1[\mbox{ for almost all }x\in\Omega.
\end{equation}

This assertion was proved by Buttazzo, Ferone, Kawohl \cite{Buttazzo1995} in 1995. We reformulate it in terms of the Legendre transformation. First, note that \eqref{eq:grad_u_not_in_01} is equivalent to

\[
	L^n\Big\{
		x\in \Omega: \exists p\in\partial u(x): 0<|p|<1
	\Big\} = 0
\]

\noindent Indeed, the set of points of non-differentiability of $u$ has zero measure.

So, by using \eqref{eq:F_0_and_L_n}, one finds that \eqref{eq:grad_u_not_in_01} follows from

\begin{equation}
\label{eq:F_0_0_l_p_l_1_eq_1}
	F_0\big(\{ p:0<|p|<1 \}|u^*\big) = 0.
\end{equation}

\begin{figure}[ht]
  \begin{center}
    \includegraphics[width=0.55\textwidth]{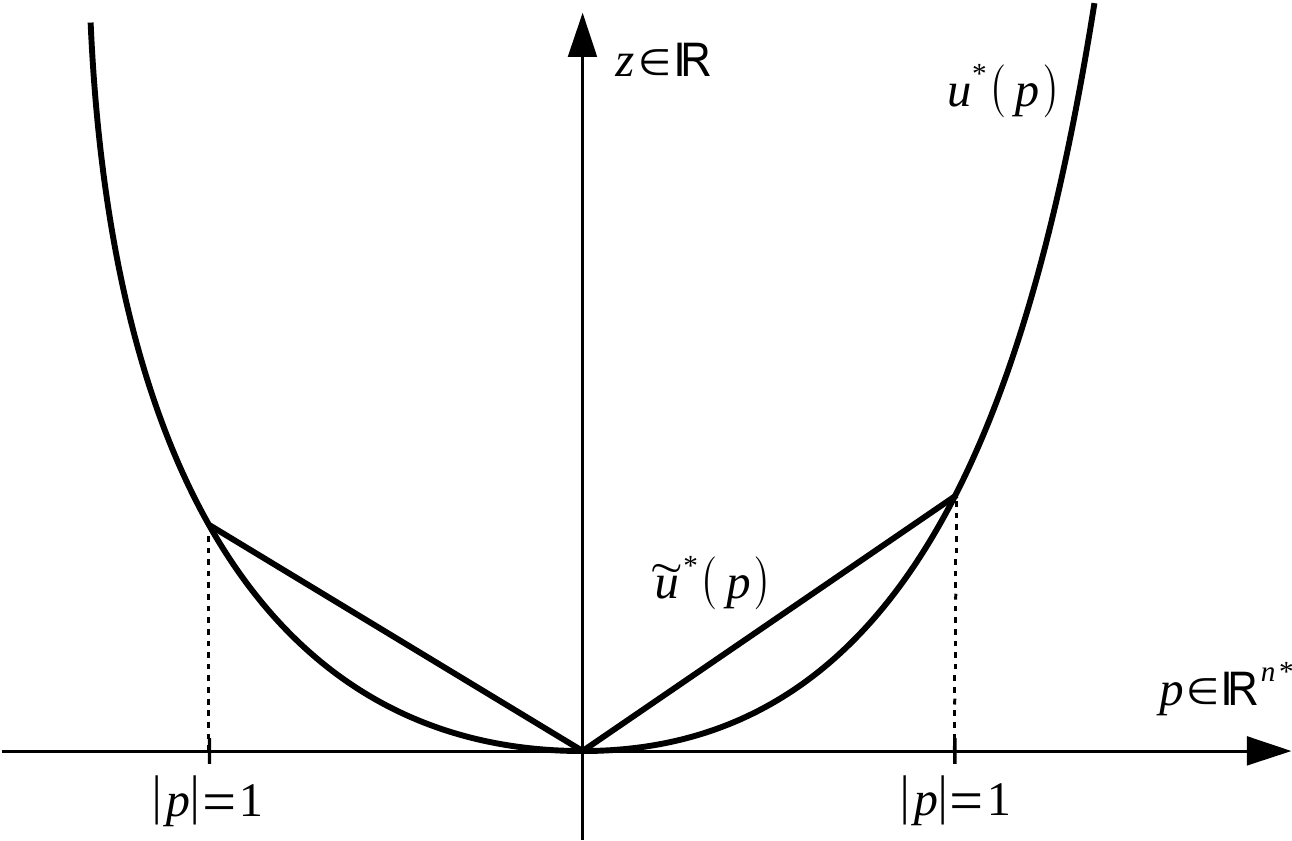}
  \end{center}
  \caption{The graph of the functions $u^*$ and $\tilde u^*$}
  \label{fig:u_and_tilde_u}
\end{figure}

Let $u^*\in C_M^*$. Consider the convex function $\tilde u^*$, which is formed by $u^*$ such that, first, $\tilde u^*(p)=u^*(p)$ for $|p|\ge 1$ and, second, the function $\tilde u^*$ is positive homogeneous for $|p|\le 1$ (see Fig. \ref{fig:u_and_tilde_u}):

\[
	\tilde u^*(p) = \conv\max\Big\{u^*(p),|p|u^*\Big(\frac{p}{|p|}\Big)\Big\}
\]

\noindent (Here one uses the natural assumption that the second argument of the maximum $|p|u^*(p/|p|)$ equals zero for $p$=0, since the continuous function $u^*(p)$ is bounded on the compact $|p|=1$).

\begin{thm}
\label{thm:grad_ge_1}
	If $u^*\in C_M^*$, then $\tilde u^*\in C_M^*$ and $J^*(\tilde u^*)\le J^*(u^*)$, and also the equality is achieved only if $\tilde u^*\equiv u^*$.
\end{thm}

Note that from Theorem \ref{thm:grad_ge_1} immediately follows the needed assertion \eqref{eq:F_0_0_l_p_l_1_eq_1}. Indeed, if $u^*\in C_M^*$  is an optimal solution, then $u^*$ coincides with $\tilde u^*$ for $|p|\le 1$; thus, \eqref{eq:F_0_0_l_p_l_1_eq_1} is realized automatically.

\medskip

\begin{proof}[Proof of Theorem \ref{thm:grad_ge_1}]
	Let us prove that $\tilde u^*\in C_M^*$. It is obvious that (i) $\tilde u^*(0)=0$, and (ii) $\tilde u^*(p)\ge u^*(p)\ge s_\Omega(p)-M$. It remains to check (iii). Let us prove that $\partial \tilde u(p)\in\Omega$ for all $p$. We know that $\tilde u^*(0)=0$. We claim that if $p\ne 0$, then $\tilde u^*(\lambda p)=u^*(\lambda p)$ for large enough $\lambda>0$. Indeed, $\tilde u^*\ge u^*$, and $|p|u^*(\frac{p}{|p|})\le u^*(p)$ for $|p|\ge 1$, since $u^*(0)=0$. So if $|p|\ge 1$, then $\max\{u^*(p),|p|u^*(\frac{p}{|p|})\}=u^*(p)$, and the operation of convex hull can only decrease a function. 

	\medskip
	
	Consequently, for any $p\ne 0$ and $\lambda>1$ large enough, we have
	\[
		\tilde u^*(p) = \tilde u^*(\frac1\lambda\lambda p) \le
			\frac1\lambda \tilde u^*(\lambda p) + 
			\frac{\lambda-1}{\lambda}\tilde u^*(0) = \frac1\lambda u^*(\lambda p) \le \frac1\lambda s_\Omega(\lambda p) = s_\Omega(p).
	\]
	
	\noindent Hence, $s_\Omega(p)-M \le \tilde u^*(p)\le s_\Omega(p)$ for any $p\ne 0$. For $p=0$ these inequalities follows from the continuity of convex functions. Denote by $\tilde u$ the Legendre--Young-Fenchel transform of $\tilde u^*$, i.e.\ $\tilde u=\tilde u^{**}$. Then $\delta_\Omega\le \tilde u\le\delta_\Omega+M$. Consequently, we have $\dom \tilde u=\Omega$
	
	So if $x\in\partial\tilde u^*(p)$, then $p\in\partial\tilde u(x)$ (since $\tilde u^*$ is a closed convex function). Since $\partial \tilde u(x)=\emptyset$ for $x\not\in\dom u=\Omega$, we have $x\in\Omega$, which proves (iii).
	
	Let us prove that $J^*(\tilde u^*)\le J^*(u^*)$. The measures $F_j(\,\cdot\,|u^*)$ and $F_j(\,\cdot\,|\tilde u^*)$ coincide on the open set $\{|p|>1\}$, since $u^*(p)=\tilde u^*(p)$ for any $|p|>1$. Therefore,
	\[
		J^*(u^*) - J^*(\tilde u^*) = 
		\int_{|p|\le 1} \frac{1}{1+|p|^2}F_0(dp|u^*) - \int_{|p|\le 1} \frac{1}{1+|p|^2}\tilde F_0(dp|\tilde u^*)
	\]
	
	\begin{figure}
	\begin{center}
		\includegraphics[width=0.4\textwidth]{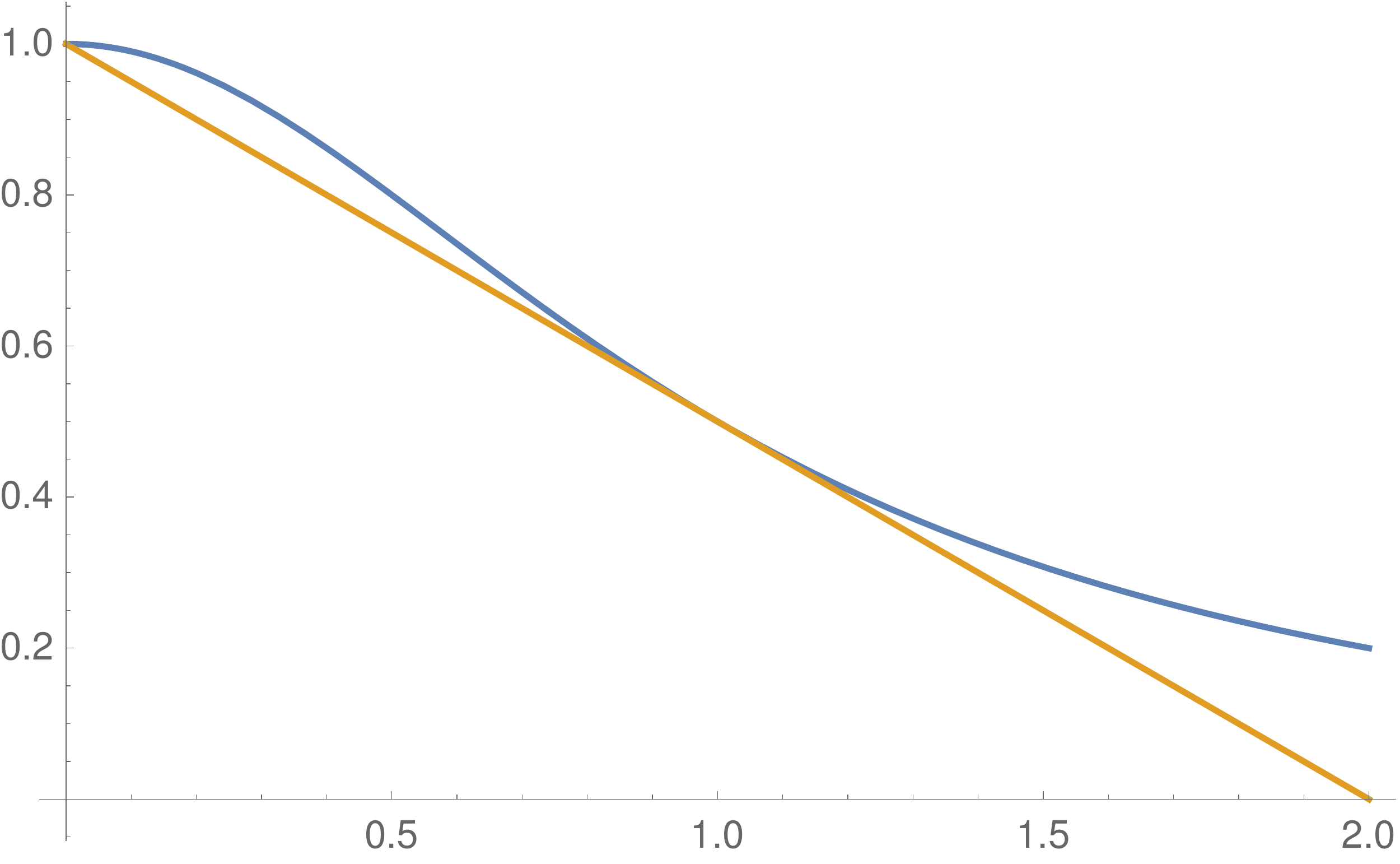}
	\end{center}
	\caption{The graphs of the functions $f(r)=1/(1+r^2)$ and $h(r)=1-r/2$}
	\label{fig:f_and_h}
	\end{figure}

	Let $f(r) = \frac{1}{1+r^2}$. Denote by $h(r)$  the linear function with $h(0)=f(0)=1$ that is tangent to the graph of the function $f$ (see Fig. \ref{fig:f_and_h}), i.e.,

	\[
		h(r) = 1-\frac12r.
	\]
	
	\noindent For brevity, we will write $f(p)=f(|p|)=\frac{1}{1+|p|^2}$ and $h(p)=h(|p|)=1-\frac12|p|$.
	
	On the one hand,
	
	\[
		\int_{|p|\le 1} f(p)F_0(dp|u^*) \ge \int_{|p|\le 1}h(p)F_0(dp|u^*),
	\]

	\noindent since $f(p)\ge h(p)$. On the other hand
	
	\[
		\int_{|p|\le 1} f(p)F_0(dp|\tilde u^*) = \int_{|p|\le 1} h(p)F_0(dp|\tilde u^*).
	\]

	\noindent Indeed, the support of the measure $F_0(dp|\tilde u^*)$ on $\{|p|\le 1\}$ is included in the union of the sphere  $\{|p|=1\}$ and the origin $\{p=0\}$, since $F_0(\{0<|p|<1\}|\tilde u^*)=0$. Hence the function $f(p)$ can be replaced by $h(p)$ (they coincide for both $p=0$ and $|p|=1$).	
	
	So
	
	\begin{multline*}
		J^*(u^*) - J^*(\tilde u^*)\ge \int_{|p|\le 1} h(p)\big[F_0(dp|u^*)-F_0(dp|\tilde u^*)\big]=\\
		=F_0(\{|p|\le 1\}|u^*)-F_0(\{|p|\le 1\}|\tilde u^*) + \frac12\int_{|p|\le 1}|p|\big[F_0(dp|\tilde u^*)- F_0(dp|u^*)\big].
	\end{multline*}

	The first difference equals 0, since
	
	\[
		F_0(\{|p|\le 1\}|u^*)-F_0(\{|p|\le 1\}|\tilde u^*) = F_0(\R^{n*}|u^*) - F_0(\R^{n*}|\tilde u^*) 
		=L^n(\Omega)-L^n(\Omega)=0
	\]

	To evaluate the integral, we return to the space $\R^n$ and use the co-area formula (by $H^{n-1}$ one denotes $(n-1)$-dimensional Hausdorff measure):
	
	\begin{multline*}
		\int_{|p|\le 1}|p|\big[F_0(dp|\tilde u^*)- F_0(dp\big|u^*)\big] = 
		\int_{\R^{n*}}|p|\big[F_0(dp|\tilde u^*)- F_0(dp|u^*)\big] = \\
		=\int_\Omega \big(|\tilde u'(x)| - |u'(x)|\big)\,dx = 
		\int_0^M \Big[H^{n-1}\big(\tilde u^{-1}(t)\big) - H^{n-1}\big(u^{-1}(t)\big) \Big]\,dt\ge 0.
	\end{multline*}

	The last inequality is fulfilled, since the sets $A_t=u^{-1}[0;t]$ and $\tilde A_t=\tilde u^{-1}[0;t]$ are convex and $A_t\subset \tilde A_t$ in view of $\tilde u\le u$. Consequently, $H^{n-1}(\partial A_t)\le H^{n-1}(\partial\tilde A_t)$.

	It remains to say that if the equality $J^*(u^*)=J^*(\tilde u^*)$ is fulfilled, then $H^{n-1}(\partial A_t)=H^{n-1}(\partial\tilde A_t)$ for almost all $t$. In this case, $A_t=\tilde A_t$ for almost all $t$ and, consequently, $u=\tilde u$.

\end{proof}

\section{Lack of strict convexity.}
\label{sec:lack_of_strict_convexity}

Let us consider the case $n=2$. In this section, we will prove a result originally proved in 1996 by Brock, Ferone, and Kawohl (see \cite{Brock1996}) and then generalized in 2001 by Lachand---Robert and Peletier (see \cite{LachandCurvature}).

\begin{thm}
\label{thm:developable_surface}
	Let $u\in C_M$ be an optimal solution of the aerodynamic Newton problem. If $\omega$ is an open subset of $\Omega$, $u\in C^2(\omega)$, and $0<u|_{\omega}<M$ then $\det u''=0$ on $\omega$.
\end{thm}

\begin{proof} 

Suppose the contrary. Let $\det u''(x)>0$ for a point $x\in\omega$. We assume by reducing $\omega$ that $\det u'' > 0$ on $\omega$ and on its neighborhood. Put $\omega^*=\{u'(x)|x\in\omega\}\subset \R^{n*}$. Using the results of Sec. \ref{sec:Legandre_transform} we have

\[
	\int_\omega \frac{1}{1+|u'(x)|^2}dx = \int_{\omega^*}\frac{1}{1+|p|^2} F_0(dp|u^*).
\]

Obviously $(u^*)'' = (u'')^{-1}$ and $\det (u^*)''>0$ on $\omega^*$. Thus $\omega^*\cap\{|p|<1\}=\emptyset$ by Theorem \ref{thm:grad_ge_1}.

Since $u^*\in C^2(\omega^*)$, using \eqref{eq:F_j_S_j} we have $F_0(dp|u^*)=\det (u^*)''\,dp_1\wedge dp_2=du^*_{p_1}\wedge du^*_{p_2}$ (where $p=(p_1,p_2)\in\R^{2*}$). Consequently,

\[
	\int_{\omega^*}\frac{1}{1+|p|^2} F_0(dp|u^*) = \int_{\omega^*}\frac{1}{1+|p|^2} du^*_{p_1}\wedge du^*_{p_2}.
\]

Let us integrate the last integral by parts. First, (recall that $f(p)=\frac{1}{1+|p|^2}$),

\[
	2f\,du^*_{p_1}\wedge du^*_{p_2} = \big( 
		f_{p_1}\, du^*_{p_2} - f_{p_2}\, du^*_{p_1}
	\big)\wedge du^* + 
	d\big(fu^*_{p_1}\,du^*_{p_2} - fu^*_{p_2}\,du^*_{p_1} \big),
\]

\noindent Second, the first term of the previous sum can be changed as follows:

\[
	f_{p_1}\, du^*_{p_2} - f_{p_2}\, du^*_{p_1} = 
	u^*_{p_1}\,df_{p_2} - u^*_{p_2}\,df_{p_1} + d(f_{p_1}u^*_{p_2}-f_{p_2}u^*_{p_1}).
\]

\noindent So by the Stokes--Poincar\'e formula, we have

\begin{multline*}
	2\int_{\omega^*} f\,du^*_{p_1}\wedge du^*_{p_2} = 
	\int_{\omega^*} \big(
		u^*_{p_1}\,df_{p_2}\wedge du^* - u^*_{p_2}\,df_{p_1}\wedge du^*
	\big) +\\
	+\int_{\partial\omega^*} \big(
		fu^*_{p_1}\,du^*_{p_2} - fu^*_{p_2}\,du^*_{p_1} + (f_{p_1}u^*_{p_2}-f_{p_2}u^*_{p_1}) du^*
	\big).
\end{multline*}

The main idea of the proof is described in the following. If $\det u''>0$ on $\omega$ then $\det (u^*)'' > 0$ on $\omega^*$. So small variations (in the space\footnote{Two times differentiable functions with compact support.} $C^2_c(\omega^*)$) of $u^*$ change neither convexity of $u^*$ nor convexity of $u=u^{**}$. The conditions $0<u|_{\omega}<M$ also remain untouched as $u=(u^*)'$. Consequently, the function $u^*$ is a weak local minimum in the minimization problem for the following multiple integral:

\begin{equation}
\label{eq:min_problem_after_int_by_parts}
	\int_{\omega^*} \big(
		v_{p_1}\,df_{p_2}\wedge dv - v_{p_2}\,df_{p_1}\wedge dv
	\big) \to\min
\end{equation}

\noindent with the following terminal constraint $v|_{\partial\omega^*} = u^*|_{\partial\omega^*}$.

Now we arrive at a contradiction by verifying that problem \eqref{eq:min_problem_after_int_by_parts} has no local minima. We claim that the Legendre condition is not fulfilled for \eqref{eq:min_problem_after_int_by_parts}. Indeed, problem \eqref{eq:min_problem_after_int_by_parts} is quadratic with respect to the derivatives $v'$:

\[
	v_{p_1}\,df_{p_2}\wedge dv - v_{p_2}\,df_{p_1}\wedge dv = 
	(v_{p_1}\ v_{p_2})
	\left(\begin{array}{rr}
		-f_{p_2p_2} &  f_{p_1p_2}\\
		 f_{p_1p_2} & -f_{p_1p_1}
	\end{array}\right)
	\left(\begin{array}{c}
		v_{p_1}\\
		v_{p_1}\\
	\end{array}\right)\,dp_1\wedge dp_2.
\]

\noindent The previous matrix has the following eigenvalues:

\[
	\lambda_1 = \frac{2}{(1+|p|^2)^2}
	\quad\mbox{and}\quad
	\lambda_2 = \frac{2}{(1+|p|^2)^3} (1-3|p|^2).
\]

\noindent Since we know that $|p|\ge1$ for all $p\in\omega^*$, we immediately have $\lambda_2<0$ and the Legendre condition is not fulfilled for problem \eqref{eq:min_problem_after_int_by_parts}.

\end{proof}

\section{Maxwell's stratum}
\label{sec:maxwell}

Let $n=2$ and $\Omega=\{|x|\le1\}$ in the present and the last sections. Here we study the class $C_M^m$ of convex bodies $u$ in $C_M$, $u|_{\partial\Omega}=M$, that are symmetric with respect to a vertical plane and have a smooth boundary outside the plane of symmetry. An optimal solution in problem \eqref{problem:start} cannot have any strictly convex smooth parts on its boundary by Theorem \ref{thm:developable_surface}. Thus, any smooth part of the boundary should be a developable surface and an optimal solution should be the convex hull of the unit circle $\{(x,M):|x|\le 1\}$ and a convex curve lying in the plane of symmetry.

Usually in the calculus of variations the term ``Maxwell's stratum'' means a set at whose points two extremals with the same value of a functional meet. Both of the extremals lose their optimality after intersection with the Maxwell stratum. If there is a symmetry, then the Maxwell stratum appears naturally when an extremal intersects its own image. We consider height in the aerodynamic Newton problem as an analog of a functional in the calculus of variations and the generating lines of developable surfaces as extremals. Thus, if a convex body is symmetric with respect to a vertical plane and is smooth everywhere (except for points on the plane), then the generating lines of two symmetrical developable surfaces intersects at points of this plane. Having in mind the above analogy, we will use the term ``Maxwell's stratum'' for the intersection of the boundary of a symmetric convex body with its plane of symmetry.

Let us now give the precise definition of the class $C_M^m$ in terms of convex analysis. Denote by $\delta_\Omega$ the indicator function of the set $\Omega$:

\begin{figure}
  \begin{center}
    \includegraphics[width=0.38\textwidth]{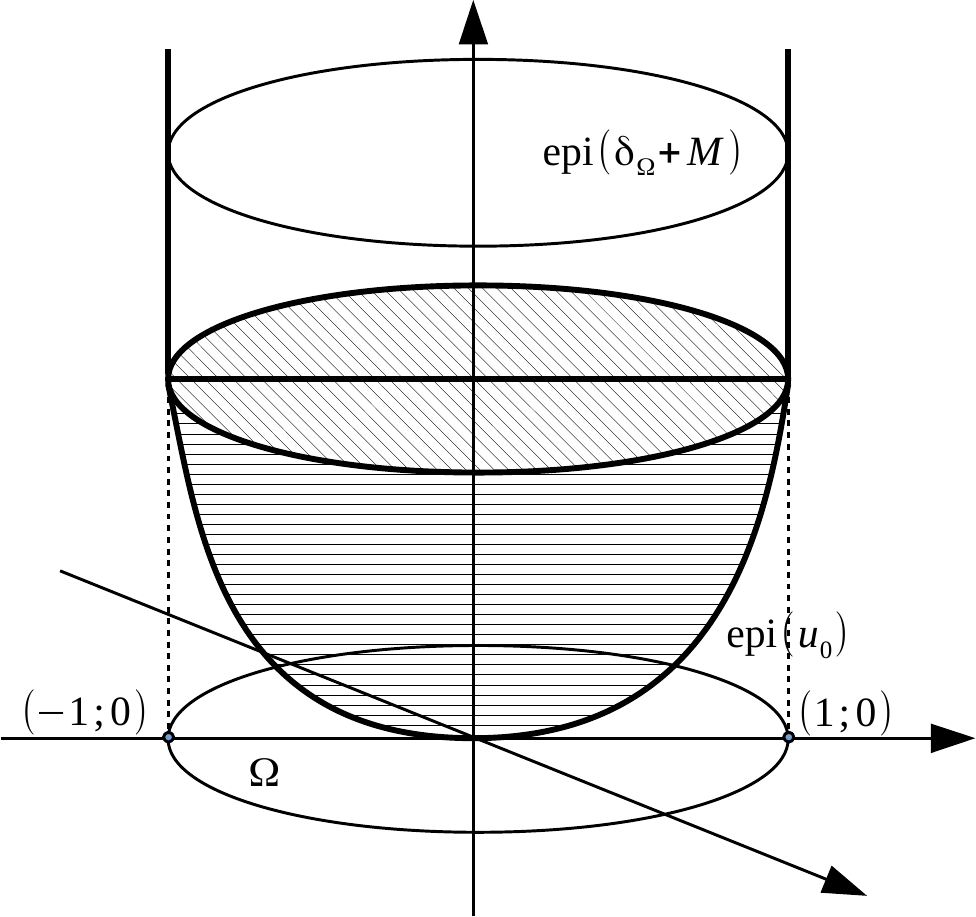}
  \end{center}
  \caption{The epigraph of the function $u = (\delta_\Omega + M)\vee u_0$ is the convex hull of the union of the epigraphs $\epi(\delta_\Omega + M)$ and $\epi u_0$.}
  \label{fig:maxwell}
\end{figure}

\[
	\delta_\Omega (x)= \left\{\begin{array}{ll}
		0&\mbox{if }x\in\Omega;\\
		+\infty&\mbox{if }x\not\in\Omega.
	\end{array}\right.
\]

\noindent If $v_1$ and $v_2$ are convex functions, then we denote by $v_1\vee v_2$ the following convex function:

\[
	(v_1\vee v_2)(x) = \inf\{(1-\lambda)v_1(x_1) + \lambda v_2(x_2)|(1-\lambda)x_1 + \lambda x_2 = x\mbox{ and }0\le\lambda\le 1\}.
\]

\noindent Note that if the convex functions $v_1$ and $v_2$ are closed and have compact effective domains, then the convex function $v_1\vee v_2$ is also closed, has compact effective domain, and $\epi (v_1 \vee v_2) = \conv(\epi v_1 \cup \epi v_2)$).

Let the vertical plane containing the line $\{x_2=0\}$ be the plane of symmetry of the body. Then the class $C_M^m$ can be described in the following way (see Fig. \ref{fig:maxwell}):

\[
	C_M^m = \{u\in C_M: u = (\delta_\Omega + M)\vee u_0\}\subset C_M,
\]

\noindent where the Maxwell stratum $u_0$ is a convex function equal to $+\infty$ outside the segment $[-1;1]$ on the line $\{x_2=0\}$ and $0\le u_0(x_1,0)\le M$ for $x_1\in[-1;1]$. We have $\min u_0=0$ by the statement of problem \eqref{problem:start} 

\begin{remark}
	Let us remark that if $M<1$, then the minimum of the functional $J$ in $C_M$ does not belong to $C_M^m$ by Remark \ref{rm:non_empty_heel}, since any optimal solution $u\in C_M$ should have the front part with nonempty interior in this case. However, for $M\ge1$, the class $C_M^m$ is of great interest. For example, in \cite{LachandPolygon}, minima of $J$ for all $M$ in the class $C_M^h$ of convex bodies with smooth side boundaries were found, and if $M>1.17953\ldots$ then any minimum in $C_M^h$ belongs to $C_M^m$. So
	
	\[
		\inf_{u\in C_M^m} J(u) \le \inf_{u\in C_M^h} J(u)
		\quad\mbox{if}\quad
		M>1.17953\ldots
	\]

\end{remark}

Using Theorems \ref{thm:pointwise_limit} and \ref{thm:C_M_compact}, it is easy to see that the functional $J$ reaches its minimum on $C_M^m$. Indeed, if $u_k\in C_M^m$ goes pointwise on $\Int\Omega$ to $u\in C_M$, then $u$ also belongs to $C_M^m$. A very good numerical results for the class $C_M^m$ (and many others) can be found in \cite{Wachsmuth}.

In this section, we will construct  an Euler-Lagrange equation for the convex conjugate function to $u_0$. Namely,

\begin{thm}
\label{thm:maxwell}
	Let $u$ be an optimal solution in $C_M^m$, $u=(\delta_\Omega+M)\vee u_0$. Suppose the function $u_0(x_1,0)$ is continuously twice differentiable on an interval $x_1\in[\alpha;\beta]$, $-1\le\alpha<\beta\le1$. If $\frac{\partial^2}{\partial x_1^2}u_0>0$ on $]\alpha;\beta[$, then the function $v(p_1)=u_0^*(p_1,0)+M$ satisfies on $]u_0'(\alpha);u_0'(\beta)[$ the following equation that does not depend on the parameter $M$:
	
	\[
		v'' = \frac{v-p_1 v'}{p_1^2-v^2}+
		\frac{2 v v'^2}{v^2+1}+\frac{v (v'^2-1)}{2 (p_1^2-v^2)}.
	\]
\end{thm}

We will consider pass to the dual space and use Hessian measures to prove the theorem. Let

\[
	\epi(v_1\wedge v_2) \eqdef \epi v_1\cap \epi v_2
	\quad\Leftrightarrow\quad
	(v_1\wedge v_2)(x) \eqdef \max\{v_1(x),v_2(x)\}.
\]

\noindent It is well known than $(v_1\vee v_2)^* = v_1^*\wedge v_2^*$ (for any proper convex functions $v_1$ and $v_2$; see \cite[chapter 3, \S 5, Theorem 16.5]{Rockafellar}).

As $\Omega=\{|x|\le 1\}$ we have

\begin{equation}
\label{eq:max_p_u_1}
	u^*(p) = \max\{|p|-M,u_0^*(p)\}
\end{equation}

Since the function $u_0$ equals $+\infty$ outside the line $\{x_2=0\}$, the function $u_0^*(p)$ depends only on $p_1$ and does not depend on $p_2$. So, by theorem~\ref{thm:Legandre_transform}, $u_0^*(0)=0$, and if $|p|-M < u_0^*(p)$, then $u_0^*(p)\ge|p|-M$ and $|\partial_{p_1}u_0^*(p)|\le 1$.

Let us calculate the resistance $J$ by Theorem \ref{thm:Legandre_transform}. We start with the calculation of the Hessian measures:

\begin{lemma}
\label{lm:F2_max}
	Let $v_0$ and $v_1$ be convex functions on $\R^{2*}$. Then if $v_0(p)>v_1(p)$ then $F_0(\,\cdot\,|v_0\wedge v_1)=F_0(\,\cdot\,|v_0)$ in a neighborhood of $p$. Similarly, if $v_0(p)<v_1(p)$, then $F_0(\,\cdot\,|v_0\wedge v_1)=F_0(\,\cdot\,|v_1)$ in a neighborhood of $p$. Finally, if $v_0(p)=v_1(p)$ on a (Lipschitz) curve $\gamma$, and $v_0,v_1\in C^2$ in a neighborhood of $\gamma$, then the Hessian measure $F_0$ of $v_0\wedge v_1$ on $\gamma$ is given by the following formula:
	
	\[
		F_0(dp|v_0\wedge v_1)|_{\gamma} = \frac12 (v_0''+v_1'')[A_{\frac\pi2}(v_1'-v_0'),dp],
		\quad
		A_\theta=\left(\begin{array}{rr}
			\cos\theta&-\sin\theta\\
			\sin\theta& \cos\theta
		\end{array}\right),
	\]

	\noindent (the direction of motion on $\gamma$ must be taken so that the domain $v_0<v_1$ is on the right of $\gamma$, and the domain $v_0>v_1$ is on the left).
\end{lemma}

\begin{proof}
	Let us parametrize $\gamma$ by a parameter $s$, $p(s)\in\gamma$ (the parameter $s$ should be taken from the statement of the lemma). Let $\eta=\{p(s):s\in[s_0,s_1]\}$ be an arc of $\gamma$. Then
	
	\[
		\eta^\varepsilon = \Big\{
			p(s)+\varepsilon\big((1-t)v'_0(p(s)) + tv'_1(p(s))\big),
			s\in[s_0,s_1], t\in[0;1]
		\Big\}.
	\]

	The above formula gives us the coordinates $(t,s)$ on $\eta^\varepsilon$ (the bijectivity of the map $[0;1] \times [s_0,s_1]\to\eta^\varepsilon$ follows from the convexity of $v_0$ and $v_1$). So
	
	\begin{multline*}
		L^2(\eta^\varepsilon) = 
		\int_{s_0}^{s_1}\int_0^1 \det
		\Big[\varepsilon(v'_1-v'_0)\quad\dot p + \varepsilon (1-t)v''_0\dot p + \varepsilon tv''_1\dot p
		\Big]dt\wedge ds =\\
		=\int_{s_0}^{s_1}\int_0^1
		\big\langle 
			\varepsilon A_{\frac\pi2}(v'_1-v'_0),
			\dot p + \varepsilon (1-t)v''_0\dot p + \varepsilon tv''_1\dot p
		\big\rangle\,dt\wedge ds = \\
		=\int_{s_0}^{s_1} \Big(
			\varepsilon\langle A_{\frac\pi2}(v'_1-v'_0),\dot p\rangle +
			\frac12\varepsilon^2(v''_0+v''_1)[A_{\frac\pi2}(v'_1-v'_0),\dot p]
		\Big)\,ds
	\end{multline*}

	\noindent Thus, $F_1(dp|v_0\wedge v_1) = \frac12\langle A_{\frac\pi2}(v'_0-v'_1),d p\rangle$ on $\gamma$, and the measure $F_0$ has the form given in the statement of the lemma.
\end{proof}

Let us now prove Theorem \ref{thm:maxwell}.

\begin{proof}[Proof of Theorem \ref{thm:maxwell}] We start with the calculation of the measure $F_0$ for \eqref{eq:max_p_u_1}. Note that the Hessian measure $F_0$ is concentrated at 0  for the function $|p|$ and is identically 0 for the function $u_0^*$. So the measure $F_0(\,\cdot\,|u^*)$ for $u^*=(|p|-M)\wedge u_0^*$ is concentrated on the curve $\gamma$ given implicitly:

\[
	u_0^*(p) \stackrel{\gamma}{=} |p|-M.
\]

Let us choose the counterclockwise direction of motion on $\gamma$. Since $u_0^*(p)>|p|-M$ at $p=0$, we obtain $v_0=u_0^*$ and $v_1=|p|-M$ by using Lemma \ref{lm:F2_max} for $u^*$. Thus 

\[
	2F_0(dp|u^*) = ((u_0^*)'' + |p|'')[A_{\frac\pi2}(|p|'-(u_0^*)'),dp].
\]

So we need to calculate the second derivatives. Denote $v(p_1)\eqdef u_0^*(p)+M$. If $\xi,\eta\in\R^{2*}$, then $(u^*_0)'[\xi]=v'(p_1)\xi_1$ and $(u_0^*(p))''[\xi,\eta] = v''(p_1)\xi_1\eta_1$. Since

\[
	|p|'[\xi] = \frac{\langle p,\xi\rangle}{|p|}
	\quad\mbox{and}\quad
	|p|''[\xi,\eta] = \frac{\langle \xi,\eta\rangle}{|p|} - 
	\frac{\langle p,\xi\rangle\langle p,\eta\rangle}{|p|^3},
\]

\noindent by using $v\stackrel{\gamma}{=}|p|$, we obtain

\begin{multline*}
	2F_0(dp|u^*)|_\gamma = \frac{p_1\,dp_2 - p_2\,dp_1}{|p|^2} - \frac{v'dp_2}{|p|}
	+\frac{p_2v'(p_1\,dp_1 + p_2\,dp_2)}{|p|^3} - \frac{p_2v''\,dp_1}{|p|} =\\
	=\frac{p_1\,dp_2 - p_2\,dp_1}{v^2}-\frac{p_1^2v'\,dp_2}{v^3} + 
	\frac{p_1p_2v'\,dp_1}{v^3}-\frac{p_2v''\,dp_1}{v}.
\end{multline*}

\noindent For the last term, we can write

\[
	\frac{p_2v''\,dp_1}{v} = \frac{p_2\,dv'}{v} = 
	d\left(\frac{p_2v'}{v}\right) - \frac{v'\,dp_2}{v} + \frac{p_2(v')^2\,dp_1}{v^2}
\]

\noindent So

\[
	2F_0(dp|u^*)|_\gamma = \frac{p_1\,dp_2 - p_2\,dp_1}{v^2} + \frac{p_2v'}{v^3}(p_1\,dp_1 + p_2\,dp_2) - \frac{p_2v'^2\,dp_1}{v^2} - d\left(\frac{p_2v'}{v}\right)
\]

Since $v=|p|$ on $\gamma$, it follows that

\[
	p_1\,dp_1 + p_2\,dp_2 \stackrel{\gamma}{=}vv'\,dp_1
	\quad\mbox{and}\quad
	dp_2 = \frac{1}{p_2}(vv' - p_1)dp_1.
\]

\noindent Thus,

\[
	2F_0(dp|u^*)|_\gamma = \frac{p_1\,dp_2 - p_2\,dp_1}{v^2} - d\left(\frac{p_2v'}{v}\right) = 
	\frac{p_1v'-v}{p_2v}\,dp_1 - d\left(\frac{p_2v'}{v}\right).
\]

Hence, using theorem \ref{thm:Legandre_transform}, we obtain

\[
	J^*(u^*) = \frac12\int_\gamma \left[
		\frac{p_1v'-v}{p_2v(1+v^2)}\,dp_1 -
		\frac{1}{1+v^2}d\left(\frac{p_2v'}{v}\right)
	\right] 
\]

\noindent Let us now integrate the second term by parts:

\[
	J^*(u^*) = \frac12\int_\gamma \left[
		-\frac{2p_2(v')^2}{(1+v^2)^2} + \frac{p_1v'-v}{p_2v(1+v^2)}
	\right]\,dp_1
\]

\noindent (there are no terminal terms because $\gamma$ is a closed loop).

It remains to note that the curve $\gamma$ consists of two symmetrical arcs. The coordinate $p_2$ is positive, $p_2>0$, and $p_1$ goes backwards on the first arc. The situation on the second arc is opposite: the coordinate $p_2$ is negative, $p_2<0$, and $p_1$ goes forward. Since $p_2^2 \stackrel{\gamma}{=} v^2-p_1^2$, we have

\begin{equation}
\label{eq:J_star_maxwell}
	J^*(u^*) = \int_{A^*}^{B^*}\left[
		\frac{2\sqrt{v^2-p_1^2}(v')^2}{(1+v^2)^2} - \frac{p_1v'-v}{\sqrt{v^2-p_1^2}v(1+v^2)}
	\right]\,dp_1,
\end{equation}

\noindent where $A^*=\frac{\partial}{\partial x_1}u_0(-1,0)$ and $B^*=\frac{\partial}{\partial x_1}u_0(1,0)$.

Let us consider the problem of minimizing the functional \eqref{eq:J_star_maxwell} on $v(p_1)$. Generally speaking, this problem is not a variational problem, because the small variations of $v=u_0^*+M$ can destroy both the convexity of $v$ (i.e. the convexity of $u_0$) and the conditions $v(0)=M$ (i.e. $\inf u_0=0$), $v(p_1)\ge |p_1|$ (i.e. $u_0\le M$ on the segment $[(-1;0),(1,0)]$) and $|v'|\le 1$ (i.e. $\dom u_0\subset [(-1;0),(1,0)]$). However, if the second derivative $u_0(x_1,0)$ with respect to $x_1$ on $]\alpha;\beta[$ is positive, then the function $u_0(\cdot,0)$ on $[\alpha;\beta]$ may be equal to its extreme values $0$ and $M$ no more than 3 times. Precisely it may turn out that $u_0(\alpha,0)=u_0(\beta,0)=M$ and $u_0(x_0,0)=0$ for a point $x_0\in[\alpha;\beta]$. In this case, the point $x_0$ divides $]\alpha;\beta[$ into two intervals $]\alpha;x_0[$ and $]x_0;\beta[$, which we denote $]\alpha_i;\beta_i[$, $i=1,2$. If $u(x,0)>0$ on $]\alpha;\beta[$, then we denote $]\alpha_1;\beta_1[=]\alpha,\beta[$. The rest of the proof is similar in both cases.

Take an index $i$. Let  $[\hat\alpha_i,\hat\beta_i]\subset\ ]\alpha_i,\beta_i[$. Then $0<u_0(x_1,0)<M$ on $[\hat\alpha_i,\hat\beta_i]$. Note that the union of the described segments $[\hat\alpha_i,\hat\beta_i]$ is $]\alpha_i;\beta_i[$. The small variations (in $C^2$) of the function $u_0(x_1,0)$, $x_1\in[\hat\alpha_i,\hat\beta_i]$, destroy neither the convexity condition nor the inequalities $0<u_0(x_1,0)<M$ nor the condition $\dom u_0\subset [(-1;0),(1,0)]$ (the latter condition is not destroyed by small variations in $C^2$, because  $[\hat\alpha_i,\hat\beta_i]\subset]-1;1[$).

Denote $\alpha^*=\frac{\partial}{\partial x_1}u_0(\hat\alpha_i,0)$ and $\beta^*=\frac{\partial}{\partial x_1}u_0(\hat\beta_i,0)$. Then the small variations of $v$ in $C^2[\alpha^*,\beta^*]$ do not destroy the above conditions on $u_0$. Consequently, $v$ is a weak local minimum (for small $C^2$ variations) in the following variational problem

\[
	\int_{\alpha^*}^{\beta^*}\left[
		\frac{2\sqrt{v^2-p_1^2}(v')^2}{(1+v^2)^2} - \frac{p_1v'-v}{\sqrt{v^2-p_1^2}v(1+v^2)}
	\right]\,dp_1\to\min_v
\]

\noindent with fixed ends $v(\alpha^*)$ and $v(\beta^*)$. Note that the Legendre condition is automatically fulfilled, and thus extremals are locally optimal.

So we have a classical variational problem and $v$ is its local minimum. Thus the function $v$ satisfies the Euler--Lagrange equation on $[\alpha^*,\beta^*]$. Since $v>|p|$ on $[\alpha^*,\beta^*]$ we obtain the equation formulated in the statement of the theorem by a straightforward computation.

Finally, recall that $\alpha^*=\frac{\partial}{\partial x_1}u_0(\hat\alpha_i,0)$ and $\beta^*=\frac{\partial}{\partial x_1}u_0(\hat\beta_i,0)$, where $[\hat\alpha_i;\hat\beta_i]\subset]\alpha_i;\beta_i[$ is arbitrary.

\end{proof}

\section{The front part}
\label{sec:heel}

\begin{figure}
  \begin{center}
    \includegraphics[width=0.38\textwidth]{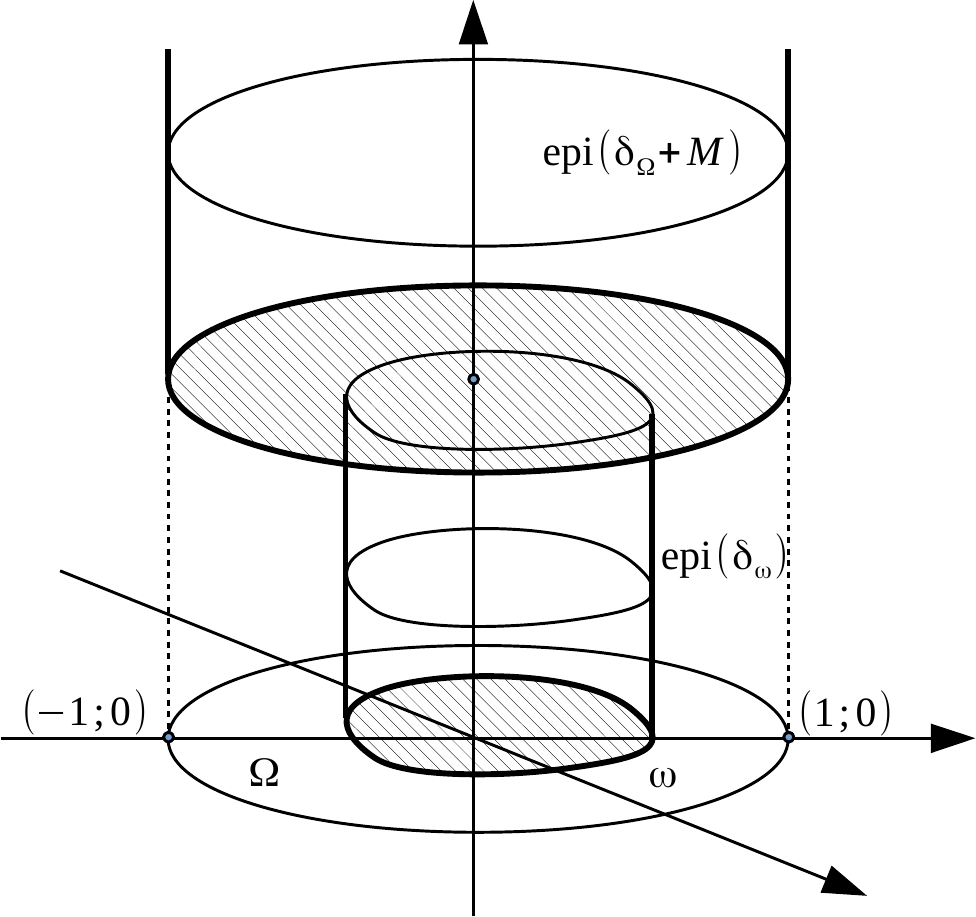}
  \end{center}
  \caption{Epigraph of the function $u = (\delta_\Omega + M)\vee \delta_\omega$ is a convex hull of the union of the epigraphs $\delta_\Omega + M$ and $\delta_\omega$.}
  \label{fig:heel}
\end{figure}

Let $n=2$. In this section, we will study the optimal form of bodies in the class $C_M^h\subset C_M$ which consists of convex bodies with smooth side boundary. Specifically  (see Fig.~\ref{fig:heel}):

\[
	C_M^h=\{u\in C_M: u=(\delta_\Omega+M)\vee \delta_\omega\},
\]

\noindent where $\omega\Subset \Int\Omega$ is a compact set. Note that the class $C_M^h$ were already studied in \cite{LachandPolygon}, where the following theorem was proved.

\begin{thm}[2001, Lachand-Robert, Peletier, \cite{LachandPolygon}]
\label{thm:heel}
	Let $u$ be an optimal solution in $C_M^h$, $u=(\delta_\Omega+M)\vee\delta_\omega$. If $\Omega\subset \R^2$ is the unit circle, then $\omega$ is a regular polygon (possibly, biangle) centered at the center of $\Omega$.
\end{thm}

It was shown in \cite{LachandPolygon} how the number of vertices of the optimal polygon and its size depend on the height $M$. The original proof of Theorem \ref{thm:heel} in \cite{LachandPolygon} may be divided in two semantic parts:

\begin{enumerate}
	\item The proof of the fact that the border of $\omega$ does not contain smooth strictly convex parts.
	\item The proof of the fact that a regular polygon with a given number of vertices is the best among all convex polygons with the same number of vertices.
\end{enumerate}

In the present section, we would like to propose a new proof of the first part by using Hessian measures; the new proof is much easier than the original one. The only thing we need is to make an absolutely straightforward computation.

So let us use Theorem \ref{thm:Legandre_transform} to prove the statement. In the dual space, we have

\[
	u^* = (\delta_\Omega+M)^* \wedge \delta_\omega^* = \max\{|p|-M,s_\omega\}.
\]

\noindent Thus, the support of the Hessian measure $F_0(\,\cdot\,|u^*)$ is the union of the origin and the curve $\gamma$, where the values of the functions $|p|-M$ and $s_\omega$ coincide. Note that $F_0(\,\cdot\,|u^*)$ on $\gamma$ determines the resistance of the side boundary of the body and $F_0(0|u^*)$ is the area of the front part $u^{-1}(0)$ and determines its resistance.

First, let us calculate the measure $F_0(\,\cdot\,|u^*)$ restricted to the curve $\gamma$. We shall use the key Lemma \ref{lm:F2_max}. So we need to compute the curve $\gamma$ where the functions $|p|-M$ and $\delta_\Omega$ coincide. Also we need to compute the first and second derivatives of the above-mentioned functions.

The functions $s_\omega(p)$ and $|p|$ are positively homogeneous. That's why we will make all calculations in polar coordinates $(r,\theta)$ on $\R^{2*}$:

\[
	v_0(p)=s_\omega(p) = rv(\theta)\quad\mbox{and}\quad v_1(p) = |p|-M = r-M.
\]

\noindent where $v(\theta)$ is a periodic Lipschitz function. Note that $\omega\Subset \Omega$, so $s_\omega(p)<s_\Omega(p)=|p|$ and, consequently, $v<1$.

The curve $\gamma$ is given in polar coordinates by the following equation:

\[
	r-M \stackrel{\gamma}{=} rv\quad\Rightarrow\quad r\stackrel{\gamma}{=}\frac{M}{1-v}.
\]

\noindent Therefore $dr=\frac{Mv'}{(1-v)^2}\,d\theta$ on $\gamma$. Thus,

\begin{equation}
\label{eq:heel_dp_gamma}
	\left(\begin{array}{c}
		dp_1\\
		dp_2
	\end{array}\right)=
	\left(\begin{array}{rr}
		\cos\theta&-r\sin\theta\\
		\sin\theta&r\cos\theta
	\end{array}\right)
	\left(\begin{array}{c}
		dr\\d\theta
	\end{array}\right)
	\stackrel{\gamma}{=}
	\frac{M}{1-v}A_\theta
	\left(\begin{array}{c}
		\frac{v'}{1-v}\\
		1
	\end{array}\right)
	d\theta.
\end{equation}

\begin{lemma}
\label{lm:r_ge_1_on_gamma}
	The curve $\gamma$ must lie outside the unit circle $\{|p|<1\}$.
\end{lemma}

\begin{proof}
	Let us use Theorem \ref{thm:grad_ge_1}. The only thing we need to check is that $\tilde u\in C_M^h$. We have $\tilde u^*(p)=\max\{|p|-M,s_\omega(p),|p|u^*(p/|p|)\}$ by construction. Since the convex function $\max\{s_\omega(p),|p|u^*(p/|p|)\}$ is positively homogeneous, it is the support function of a set $\tilde\omega\subset \R^2$. Moreover, $\tilde\omega\Subset\Omega$ as $u^*(p)<|p|$ and $s_\omega(p)<|p|$. Then $\tilde u^* = \max\{|p|-M,s_{\tilde\omega}\}$ and $\tilde u\in C_M^h$. Consequently, $u^*=\tilde u^*$ by Theorem \ref{thm:grad_ge_1}.
	
	The function $u^*$ is positively homogeneous at $|p|\le 1$. Since $u^*(p)=s_\omega^*(p)$ for all $p$ small enough, it follows that $u(p)=s_\omega(p)$ for $|p|\le 1$. So $s_\omega(p)\ge |p|-M$ for $|p|\le 1$ and, for $|p|=1$, we have

	\[
		v\ge 1-M
		\quad\Rightarrow\quad
		r\stackrel{\gamma}{=}\frac{M}{1-v}\ge 1.
	\]

\end{proof}

The advantage of our approach is seen in the simplicity of the last proof. For example, the original proof of this inequality given in \cite{LachandPolygon} takes the whole section.

Let us now compute the first and the second derivatives (on $p$) of the functions $v_{0,1}$ in polar coordinates\footnote{For brevity we write $\partial_{x}\eqdef\frac{\partial}{\partial x}$.}:

\[
	\left(\begin{array}{c}
		\partial_{p_1}\\
		\partial_{p_2}
	\end{array}\right)=
	A_\theta
	\left(\begin{array}{c}
		\partial_{r}\\
		\frac1r\partial_{\theta}
	\end{array}\right).
\]

\noindent This yields

\begin{equation}
\label{eq:heel_v_0_v_1_derivatives}
	\left(\begin{array}{c}
		\partial_{p_1}\\
		\partial_{p_2}
	\end{array}\right)v_0=
	A_\theta
	\left(\begin{array}{c}
		v\\v'
	\end{array}\right)
	\quad\mbox{and}\quad
	\left(\begin{array}{c}
		\partial_{p_1}\\
		\partial_{p_2}
	\end{array}\right)v_1=
	A_\theta
	\left(\begin{array}{c}
		1\\0
	\end{array}\right).
\end{equation}

The second derivatives may be determined by using\footnote{The symbol $T$ means transposition.} $\partial_\theta A_\theta=A_{\frac\pi2+\theta}$:

\begin{multline*}
	v''_0 = 
	\left(\begin{array}{c}
		\partial_{p_1}\\
		\partial_{p_2}
	\end{array}\right) 
	\left(\begin{array}{c}
		\partial_{p_1} v_0\\
		\partial_{p_2} v_0
	\end{array}\right)^T=
	A_\theta
	\left(\begin{array}{c}
		\partial_{r}\\
		\frac1r\partial_{\theta}
	\end{array}\right)
	\big[
		(v\ \ v')A_{-\theta}
	\big] = \\
	=\frac1r A_\theta
	\left[
		\left(\begin{array}{cc}
			0&0\\v'&v''
		\end{array}\right)A_{-\theta}-
		\left(\begin{array}{cc}
			0&0\\v&v'
		\end{array}\right)A_{\frac\pi2-\theta}
	\right] =\\
	= \frac1r A_\theta
	\left[
		\left(\begin{array}{cc}
			0&0\\v'&v''
		\end{array}\right)-
		\left(\begin{array}{cc}
			0&0\\v&v'
		\end{array}\right)
		\left(\begin{array}{cc}
			0&-1\\1&0
		\end{array}\right)
	\right] A_{-\theta}
\end{multline*}

\noindent Hence

\begin{equation}
\label{eq:heel_v_0_2der}
	v''_0 \stackrel{\gamma}{=} \frac{1-v}M (v+v'')
	A_\theta 
	\left(\begin{array}{cc}
		0&0\\0&1
	\end{array}\right)A_{-\theta}
\end{equation}

Note that $\det v''_0=0$ and $\mathrm{tr}\, v''_0=\frac{1-v}M (v+v'')$. Since $v<1$, we see that the convexity of the original function $v_0$ is equivalent to the inequality $v+v''\ge 0$ (here $v$ is a Lipschitz function, $v'\in L_\infty$, and $v''$ is a generalized function of first order).

For the second derivative of $v_1$, we have

\begin{equation}
\label{eq:heel_v_1_2der}
	v''_1 =
	A_\theta
	\left(\begin{array}{c}
		\partial_{r}\\
		\frac1r\partial_{\theta}
	\end{array}\right)
	\big[
		(1\ \ 0)A_{-\theta}
	\big] =
	-\frac1r A_\theta
	\left(\begin{array}{cc}
		0&0\\1&0
	\end{array}\right)
	A_{\frac\pi2-\theta} \stackrel{\gamma}{=}
	\frac{1-v}{M} A_\theta
	\left(\begin{array}{cc}
		0&0\\0&1
	\end{array}\right)
	A_{-\theta}.
\end{equation}

Using Lemma \ref{lm:F2_max}, and equalities \eqref{eq:heel_dp_gamma}, \eqref{eq:heel_v_0_v_1_derivatives}, \eqref{eq:heel_v_0_2der}, and \eqref{eq:heel_v_1_2der}, we obtain

\[
	F_0(dp|u^*)|_\gamma = 
	\frac12(1+v+v'')\ 
	\Big(1-v\ \ -v'\Big)
	A_{-\theta}A_{-\frac\pi2}A_\theta
	\left(\begin{array}{cc}
		0&0\\0&1
	\end{array}\right)
	A_{-\theta}A_\theta
	\left(\begin{array}{c}
		\frac{v'}{1-v}\\
		1
	\end{array}\right)
	d\theta.
\]

\noindent Finally, we need to multiply the received matrices, obtaining

\begin{equation}
\label{eq:heel_F2_gamma}
	F_0(dp|u^*)|_\gamma = \frac12(1+v+v'')(1-v)\,d\theta.
\end{equation}

Let us now compute $F_0(0|u^*)=F_0(0|s_\omega)$:

\[
	F_0(0|s_\omega) = L^2(\partial s_\omega(0)) = L^2(\omega).
\]

\noindent We determine the area of $\omega$ using the Stokes--Poincar\'e formula. Namely, since  $\omega=\partial s_\omega(0)=\conv\bigcup_{|p|=1}\partial s_\omega(p)$, from \eqref{eq:heel_v_0_v_1_derivatives} we find

\[
	\partial \omega =\Big\{
	\left(\begin{array}{cc}
		x\\y
	\end{array}\right)=
	A_\theta
	\left(\begin{array}{cc}
		v\\v'
	\end{array}\right),\mbox{ where }\theta\in[0;2\pi]
	\Big\}
\]

\noindent So

\[
	L^2(\omega) = \frac12\int_{\partial\omega}(x\,dy-y\,dx) =
	\frac12\int_0^{2\pi}\big[v^2 + vv''\big]\,d\theta.
\]

Thus, we can now determine the whole resistance of the body from \eqref{eq:heel_F2_gamma} (recall that $f(r)=\frac{1}{1+r^2}$):

\[
	J^*(u^*) = \frac12\int_0^{2\pi}\Big[
		(v^2+vv'')f(0) + (1+v+v'')(1-v)f(r)
	\Big]\,d\theta.
\]

\noindent Let us integrate by parts the terms with $v''$. Since $r\stackrel{\gamma}{=}\frac{M}{1-v}$, we have

\begin{equation}
\label{eq:J_start_heel}
	J^*(u^*) = \frac12\int_0^{2\pi}\Big[
		v^2 + (1-v^2)f(r) - (1-f(r)+rf'(r))v'^2
	\Big]\,d\theta,
\end{equation}

\noindent where

\[
	1-f(r)+rf'(r) = \frac{r^2(r^2-1)}{(1+r^2)^2}\stackrel{\gamma}{\ge}0
\]

\noindent for $r\ge1$ on $\gamma$ by Lemma \ref{lm:r_ge_1_on_gamma}.

The analogy with Sec. \ref{sec:lack_of_strict_convexity} is obvious. Let us assume the converse: the strict inequality $v+v''>0$ holds on an arc $\theta\in]\alpha;\beta[$ and $v\in C^2$. On the one hand, $v$ must be a local minimum of the functional \eqref{eq:J_start_heel} with respect to $C^2$-variations on $[\alpha;\beta]$. On the other hand, the Legendre condition is not fulfilled. Hence the statement is proved.

\medskip

So the result has the following meaning. Firs, if the equality $v(\theta)+v''(\theta)=0$ holds for $\theta\in[\alpha;\beta]$, then the corresponding arc of the border of $\omega$ consists of one point at which the border $\partial\omega$ has a fracture, since it has a non-unique support hyperplane. Second, if the border $\partial\omega$ has a straight-line segment with an angle $\theta_0$ (to the axis $Ox_2$), then, in a neighborhood of $\theta_0$, we have $v(\theta)+v''(\theta) = \lambda\hat\delta(\theta-\theta_0)$, where $\hat\delta$ is the Dirac delta function and $\lambda>0$ is determined by the length of the straight-line segment and its distance from the origin.

\bigskip 

We would like to express deep gratitude to Gerd Wachsmuth for his very important comment concerning Theorem \ref{thm:maxwell}.

\bibliographystyle{abbrv}

\begin{thebibliography}{10}

\bibitem{Plakhov}
A.~Aleksenko and A.~Plakhov.
\newblock Bodies of zero resistance and bodies invisible in one direction.
\newblock {\em Nonlinearity}, 22(6):1247--1258, 2009.

\bibitem{Brock1996}
F.~Brock, V.~Ferone, and B.~Kawohl.
\newblock A symmetry problem in the calculus of variations.
\newblock {\em Calculus of Variations and Partial Differential Equations},
  4(6):593--599, Oct 1996.

\bibitem{Buttazzo1995}
G.~Buttazzo, V.~Ferone, and B.~Kawohl.
\newblock Minimum problems over sets of concave functions and related
  questions.
\newblock {\em Mathematische Nachrichten}, 173(1):71--89, 1995.

\bibitem{ColesantiFirst}
A.~Colesanti.
\newblock A steiner type formula for convex functions.
\newblock {\em Mathematika}, 44(1):195--214, 1997.

\bibitem{ColesantiHug}
A.~Colesanti and D.~Hug.
\newblock Hessian measures of semi-convex functions and applications to support
  measures of convex bodies.
\newblock {\em Manuscripta Mathematica}, 101:209--238, 2000.

\bibitem{Guasoni}
P.~Guasoni.
\newblock {\em Problemi di ottimizzazione di forma su classi di insiemi
  convessi}.
\newblock Test di Laurea, Universita di Pisa, 1995-1996.

\bibitem{Joly}
J.-L. Joly.
\newblock Une famille de topologies sur l'ensemble des fonctions convexes pour
  lesquelles la polarit\'e est bicontinue.
\newblock {\em J. Math. Pures Appl.}, 52(9):421--441, 1973.

\bibitem{LachandCurvature}
T.~Lachand-Robert and M.~Peletier.
\newblock An example of non-convex minimization and an application to
  {Newton}'s problem of the body of least resistance.
\newblock {\em Annales de l'Institut Henri Poincare (C) Non Linear Analysis},
  18(2):179--198, 2001.

\bibitem{LachandPolygon}
T.~Lachand-Robert and M.~Peletier.
\newblock Newton's problem of the body of minimal resistance in the class of
  convex developable functions.
\newblock {\em Mathematische Nachrichten}, 226(1):153--176, 2001.

\bibitem{Marcellini1990}
P.~Marcellini.
\newblock {\em Non convex integrals of the Calculus of Variations}, pages
  16--57.
\newblock Springer Berlin Heidelberg, Berlin, Heidelberg, 1990.

\bibitem{Mosco}
U.~Mosco.
\newblock On the continuity of the young-fenchel transform.
\newblock {\em Journal of Mathematical Analysis and Applications}, 35(3):518 --
  535, 1971.

\bibitem{Rockafellar}
R.~T. Rockafellar.
\newblock {\em Convex Analysis}.
\newblock Princeton University Press, Princeton, 1997.

\bibitem{Wets}
G.~Salinetti and R.~J. Wets.
\newblock On the relations between two types of convergence for convex
  functions.
\newblock {\em Journal of Mathematical Analysis and Applications}, 60(1):211 --
  226, 1977.

\bibitem{Schneider}
R.~Schneider.
\newblock {\em Convex Bodies: The Brunn-Minkowski Theory}.
\newblock Cambridge University Press, Cambridge, second expanded edition
  edition, 2014.

\bibitem{Wachsmuth}
G.~Wachsmuth.
\newblock The numerical solution of {Newton}'s problem of least resistance.
\newblock {\em Mathematical Programming}, 147(1):331--350, Oct 2014.

\end{thebibliography}

\end{document}